\documentclass[12pt]{amsart}

\usepackage{amsmath}
\usepackage{amsthm}
\usepackage{amssymb}
\usepackage{amscd}
\usepackage{setspace}
\usepackage[headings]{fullpage}
\usepackage{mathrsfs}
\input xy
\xyoption{all}
\usepackage[colorlinks=true]{hyperref}	%for hyperlinks
\hypersetup{urlcolor=blue}		%set color of links to blue

\newtheorem{theorem}{Theorem}
\newtheorem{proposition}[theorem]{Proposition}
\newtheorem{lemma}[theorem]{Lemma}
\newtheorem{corollary}{Corollary}[theorem]

\DeclareMathOperator{\gl}{GL}

\newcommand{\ZZ}{\mathbb{Z}}
\newcommand{\RR}{\mathbb{R}}
\newcommand{\CC}{\mathbb{C}}
\newcommand{\OO}{\mathcal{O}}

\newcommand{\mtrxfo}[8]{\left(\begin{array}{cccc}#1 & #2 & #3 & #4 \\ #5 & #6 & #7 & #8 \\ }
\newcommand{\mtrxft}[8]{ #1 & #2 & #3 & #4 \\ #5 & #6 & #7 & #8 \end{array}\right)}

\DeclareMathOperator{\Sym}{Sym}

\DeclareMathOperator{\Tr}{Tr}

\DeclareMathOperator{\go}{GO}
\DeclareMathOperator{\disc}{\!Disc}
\DeclareMathOperator{\vol}{Vol}

\newcommand{\mc}[1]{\mathcal{#1}}

\begin{document}
\title[Equidistribution of shapes of rings of integers]{The Equidistribution of Lattice Shapes of Rings of Integers in Cubic, Quartic, and Quintic Number Fields}
\author{Manjul Bhargava and Piper Harron}

\date{\today}
\maketitle

\begin{abstract}
For $n=3$, $4$, and $5$, we prove that, when $S_n$-number fields of degree $n$ are ordered by their absolute discriminants, the lattice shapes of the rings of integers in these fields become equidistributed in the space of  lattices. 
%We prove that, when cubic, $S_4$-quartic, or quintic fields are ordered by their absolute discriminants, the lattice shapes of the rings of integers in these fields become equidistributed in the space of  lattices. 
\end{abstract}
\tableofcontents

\section{Introduction}
Let $K$ be a number field of degree $n$ and $\mathcal{O}_K$ its ring of integers.  Then $\mathcal{O}_K$ can be embedded in $\mathbb{R}^n$ by 
\begin{equation}\label{embedding}
 x \mapsto (\sigma_1(x),
\ldots,\sigma_r(x), \tau_1(x),\ldots,\tau_s(x)) \in \mathbb{R}^r \times \mathbb{C}^s \cong \mathbb{R}^n
\end{equation}
where $\sigma_1, \ldots, \sigma_r$ denote the $r$ real embeddings and $\tau_1, \tau_1',\ldots, \tau_s,\tau_s'$ the $s$ pairs of complex embeddings of $K.$  This endows $\OO_K$ with a natural positive definite quadratic form $q$, namely, for $x \in \mathcal{O}_K$, we define 
\begin{equation}\label{qdef1}
q(x):= \sigma_1(x)^2 + \cdots + \sigma_r(x)^2 + 2|\tau_1(x)|^2 + \cdots + 2|\tau_s(x)|^2.
\end{equation}
 The square of the covolume of the lattice in $\RR^n$ given by the image of $\OO_K$ (equivalently, the determinant of the Gram matrix of the quadratic form $q$) is given by the absolute value~$|\disc(K)|$ of the discriminant of the number field $K$.
When $K$ is totally real, the form $q(x)$ coincides with the usual {\it trace form} $\mathrm{Tr}(x^2)$ on $\OO_K$.

The \emph{shape} of $\mathcal{O}_K$ is defined to be the $(n-1)$-ary quadratic form, up to scaling by $\RR^\times$,  obtained by restricting $q$ to $\{ x\in \ZZ+n\mathcal{O}_K : \Tr_\mathbb{Q}^K(x)=0\},$ which is therefore well-defined up to the action of $\mathbb{G}_m(\RR)\times \gl_{n-1}(\mathbb{Z})$.
Alternatively, the shape of $\mathcal{O}_K$ may be defined as the $(n-1)$-ary quadratic form, up to scaling by $\RR^\times$, obtained by restricting the real quadratic form $q$ on $K$ (as defined by (\ref{qdef1})) to the projection of $\mathcal{O}_K$ onto the hyperplane in~$K$ that is orthogonal to~1.  (It is convention to define the shape of $\mathcal O_K$  in terms of the lattice orthogonal to $\mathbb Z$ in $\mathcal O_K$, because $\mathbb Z$ is always a subring of $\mathcal O_K$, while the orthogonal complement of $\mathbb Z$ gives the ``new'' part of the lattice.)  
Hence the shape of $\mathcal O_K$ may be viewed as an element of $$\mathcal{S}_{n-1} := \gl_{n-1}(\mathbb{Z})\backslash \gl_{n-1}(\mathbb{R}) / \go_{n-1}(\mathbb{R})
% \cong \gl^{\pm1}_{n-1}(\mathbb{Z}) \backslash \gl^{\pm1}_{n-1}(\mathbb{R}) / \go^{\pm1}_{n-1}(\mathbb{R})
 $$ which we call the \emph{space of shapes} of lattices of rank $n-1$.  There is a natural measure $\mu$ on $\mathcal{S}_{n-1}$ obtained from the Haar measure on $\gl_{n-1}(\mathbb{R})$ and $\go_{n-1}(\mathbb{R}),$ and it is a classical result of Minkowski~\cite[(85)]{Minkowski} that $\mu(\mathcal{S}_{n-1}) < \infty.$ 

In \cite{Terr}, Terr showed that when cubic fields of any given signature are 
ordered by absolute discriminant, then the shapes of the rings of integers in these 
fields become equidistributed in $\mathcal{S}_2$ with respect to $\mu$. The purpose 
of this paper is to prove the analogue of this statement also for quartic and 
quintic number fields:

\begin{theorem}\label{thm:main}
Let $n=3$, $4$, or $5$.  When isomorphism classes of $S_n$-number fields of degree~$n$ and any given signature are ordered by absolute discriminant, the shapes of the rings of integers in these fields become equidistributed in $\mc{S}_{n-1}$ with respect to $\mu$.
%	\mbox{}
%	\begin{enumerate}
%		\item {\rm(\cite[Theorem 11.7]{Terr})} When isomorphism classes of cubic fields of any given signature are ordered by absolute discriminant, the shapes of the rings of integers in these fields become equidistributed in $\mc{S}_2$ with respect to $\mu$.
%		\item When isomorphism classes of $S_4$-quartic fields of any given signature are ordered by absolute discriminant, the shapes of the rings of integers in these fields become equidistributed in $\mc{S}_3$ with respect to $\mu$.
%		\item When isomorphism classes of quintic fields of any given signature are ordered by absolute discriminant, the shapes of the rings of integers in these fields become equidistributed in $\mc{S}_4$ with respect to $\mu$.
%	\end{enumerate}
\end{theorem}

More precisely, for $n=3$, $4,$ or $5$, let $N_n^{(i)}(X)$ denote the number of 
isomorphism classes of $n$-ic fields having $i$ pairs of complex embeddings, 
associated Galois group $S_n$, 
and absolute discriminant less than $X$.  Also, for a measurable subset 
$W\subseteq\mc{S}_{n-1}$ whose boundary has measure $0$, let $N_n^{(i)}(X,W)$ denote 
the number of isomorphism classes of $n$-ic fields having $i$ pairs of complex 
embeddings, associated 
Galois group $S_n$, absolute discriminant less than $X$, and ring of integers with 
shape in $W$. 
%(and, when $n=4$, whose Galois group is $S_4$). 
Then, we prove that 
\begin{equation}\label{toprove} 
\lim_{X\rightarrow\infty}\frac{N_n^{(i)}(X,W)}{N_n^{(i)}(X)}=\frac{\mu(W)}{\mu(\mc{S}_{n-1})}. 
\end{equation}

The condition that the associated Galois group be $S_n$ may be dropped in Theorem~\ref{thm:main} in the cases $n=3$ and $n=5$, since 100\% of all cubic fields (resp.\ quintic fields), when ordered by discriminant, have associated Galois group $S_3$ (resp.\ $S_5$).  However, the condition is needed in the case $n=4$, as the Galois group $S_4$  does {\it not} occur with density 1 among all quartic fields when ordered by discriminant.  Indeed, about $9.356\%$ of all quartic fields have associated Galois group $D_4$ rather than $S_4$, and the lattice shapes of the rings of integers in $D_4$-quartic fields cannot be equidistributed, as is easily seen.  For example, note that
if $K$ is a $D_4$-quartic field, then $K$ has a nontrivial automorphism of order $2$ which means that $\mc{O}_K$ does too, as does its underlying lattice.
It is an interesting problem to determine the distribution of lattice shapes for $n$-ic number fields having a given non-generic (i.e., non-$S_n$) associated Galois group, even heuristically.
For the simple answer in the case of $C_3$-cubic number fields, and related results, see \cite{BhSh}.
In the general case of associated Galois group $S_n$, we naturally conjecture that Theorem~\ref{thm:main} is true for all values of $n$.

Our proof of Theorem~\ref{thm:main} is uniform for $n\in\{3,4,5\}$.  It relies on the existence of parametrizations of cubic, quartic, and quintic orders by the orbits of an algebraic group on a representation 
%defined over $\ZZ$ 
as established in \cite{DF,III,IV}, and the corresponding counting results of \cite{DavI, DavII, DD4,DD5}, together with certain geometry-of-numbers and sieving arguments and volume computations.  In particular, our method yields a considerably simpler proof of Terr's result (which is the case $n=3$).

This article is organized as follows.  In Section 2, we describe in a uniform manner the above parametrizations of cubic, quartic, and quintic orders from \cite{DF,III,IV}.  In Section 3, we show how certain geometry-of-numbers considerations, in conjunction with the counting results of \cite{DavI, DavII, DD4,DD5}, yield expressions for the number of orders of bounded discriminant having lattice shape in a given subset~$W$ of lattice shapes.  The corresponding results for orders satisfying any finite set of local conditions at finitely many primes are then discussed in Section 4.  In Section~5, this is used, via a sieve, to derive analogous expressions for the number of {\it maximal} orders of bounded discriminant having lattice shape in a given subset $W$ of lattice shapes.  These expressions are given in terms of volumes of certain regions in Euclidean space.  The ratios of these volumes are then computed in Section 6, yielding Theorem~\ref{thm:main}.

\section{Preliminaries}

The key algebraic ingredient in proving Theorem \ref{thm:main} for cubic, quartic, and quintic fields is the parametrizations of cubic, quartic, and quintic orders in \cite{DF, III, IV}. Let us fix the degree $n\in\{3,4,5\}$ of number fields we are considering, and for any ring $T$, let $V_T$ denote
\begin{enumerate}
	\item the space $\Sym^3T^2(\otimes T)$ of binary cubic forms over $T$, if $n=3$;
	\item the space $\Sym^2T^3\otimes T^2$ of pairs of ternary quadratic forms over $T$, if $n=4$; and
	\item the space $T^4\otimes\wedge^2T^5$ of quadruples of alternating quinary $2$-forms over $T$, if $n=5$.
\end{enumerate}
For $n=3$, $4$, or $5$, we set $r=r(n)=2$, $3$, or $6$, respectively.  Then the group $G_T=\gl_{n-1}(T)\times \gl_{r-1}(T)$ acts naturally on $V_T$.
%, where $G_T=\gl_2(T),$ $\gl_2(T)\times\gl_3(T)$, or $\gl_4(T)\times\gl_5(T)$ in accordance with whether $n=3,4$, or $5$. 
The ring of polynomial invariants for the action of $G_\ZZ$ on $V_\ZZ$ is generated by one element, called the \emph{discriminant}, which has degree $d=4$, $12$, or $40$ for $n=3,$ $4$, or $5$, respectively (see \cite{SK}). (Note that $d$ also gives the rank of the free $T$-module $V_{T}$.) We say that an element of $V_T$ is \emph{nondegenerate} if its discriminant is nonzero.  

We then have the following theorem parametrizing rings of rank $n$, which follows from \cite[\S 15]{DF}, \cite[
%Theorem 1, 
Corollary 5]{III}, \cite[
%Theorem 1, 
Corollary 3]{IV}:
%We first state a {\it based} version of the theorem.

\begin{theorem}\label{par}
The nondegenerate elements of $V_\ZZ$ are in canonical bijection with isomorphism classes of pairs $((R,\alpha),(S,\beta))$, where $R$ is a nondegenerate ring of rank $n$ and $S$ is a rank $r$ resolvent ring of $R$, and $\alpha$ and $\beta$ are $\ZZ$-bases for $R/\ZZ$ and $S/\ZZ$, respectively. In this bijection, the discriminant of an element of $V_\ZZ$ is equal to the discriminant of the corresponding ring $R$ of rank $n$. 
Moreover, under this bijection, the action of $G_\ZZ$ on $V_\ZZ$ results in a corresponding natural action of $G_\ZZ
=\gl_{n-1}(\ZZ)\times \gl_{r-1}(\ZZ)$ 
on $(\alpha,\beta)$. 
Finally, every isomorphism class of maximal ring $R$ of rank $n$ arises in this bijection, and the elements of $V_\ZZ$ yielding $R$ consists of exactly one $G_\ZZ$-orbit.
\end{theorem}

A {\it ring of rank $n$} is any ring that is free of rank $n$ as a $\ZZ$-module. 
We say that a ring of rank $n$ is {\it nondegenerate} if it has nonzero discriminant.
Rings of rank $2$, $3$, $4$, $5$, or $6$ are called {\it quadratic}, {\it cubic}, {\it quartic}, {\it quintic}, and {\it sextic} rings, respectively. 
A \emph{resolvent ring} of a cubic, quartic, or quintic ring is a quadratic, cubic, or sextic ring, respectively, satisfying certain conditions and whose precise definition will not be needed here (see \cite{III,IV} for details).  
%The action of $\gl_{n-1}(\ZZ)\times \gl_{r-1}(\ZZ)$ on the pair $(\alpha,\beta)$ of bases is as follows:
%$$
%(g_{n-1},g_{r-1})\cdot (\alpha,\beta) =
%\begin{cases} 
%\left(|\!\det g_{n-1}|^{1} |\!\det g_{r-1} |^{1}g_{n-1}\cdot\alpha\,,\,
%|\!\det g_{n-1}|^{3} |\!\det g_{r-1} |^{1}g_{r-1}\cdot\beta
%\right) & \text{if $n=3$;} 
%\\ 
%\left(|\!\det g_{n-1}|^{1} |\!\det g_{r-1} |^{1}g_{n-1}\cdot\alpha\,,\,
%|\!\det g_{n-1}|^{2} |\!\det g_{r-1} |^{1}g_{r-1}\cdot\beta
%\right) & \text{if $n=4$;} 
%\\ 
%\left(|\!\det g_{n-1}|^{1} |\!\det g_{r-1} |^{2}g_{n-1}\cdot\alpha\,,\,
%|\!\det g_{n-1}|^{3} |\!\det g_{r-1} |^{4}g_{r-1}^\ast\cdot\beta
%\right) & \text{if $n=5$,} 
%\end{cases}
%$$
%where we have used $g^\ast$ to denote the cofactor matrix of $g$ (i.e., the matrix $|\!\det g|(g^{-1})^t$).  These natural actions can be deduced from ?.

We say that an element $x\in V_\mathbb{Z}$ is \emph{irreducible} if in the corresponding pair $(R,S)$, the ring $R$ is isomorphic to an order in an $S_n$-field of degree $n.$  In the next section, we will be interested in counting irreducible elements in $V_\ZZ$, up to $G_\ZZ$-equivalence, having bounded discriminant.

%Theorem~\ref{par} holds not only over $\ZZ$ but over more general base rings $T$ (reference?).  
We note that Theorem~\ref{par} also holds with any field $K$ in place of $\ZZ$, with the same proofs as in~\cite{III,IV} (see also \cite{WY} for earlier results of this type).  We will require here the analogue of Theorem~\ref{par} over $\RR$:
% It will be convenient to state a {\it based} version of the theorem:
%A {\it based ring or rank $n$ over $T$} is any ring $R$ that is free of rank $n$ over $T$ and is equipped with a
%basis $\bar \alpha_1,\ldots,\bar \alpha_{n-1}$ 

\begin{theorem}\label{par2}
There is a canonical bijection between the nondegenerate elements of $V_\RR$ and isomorphism classes of pairs $((R,\alpha),(S,\beta))$, where $R$ is a nondegenerate ring of rank $n$ over~$\RR$ and $S$ is the $($unique$)$ rank $r$ resolvent ring of $R$ over $\RR$, and $\alpha$ and $\beta$ are $\RR$-bases for $R/\RR$ and $S/\RR$, respectively.
Moreover, under this bijection, the action of $G_\RR$ on $V_\RR$ results in the corresponding natural action of $G_\RR
=\gl_{n-1}(\RR)\times \gl_{r-1}(\RR)$ 
on $(\alpha,\beta)$. 
%In this bijection, the discriminant of an element of $V_\RR$ is equal to the discriminant of the corresponding ring 
%$R$ of rank $n$. 
\end{theorem}
Theorems~\ref{par} and \ref{par2} are compatible with each other under the inclusion $V_\ZZ\subset V_\RR$: if the ring associated to $v\in V_\ZZ$ via Theorem~\ref{par} is $R$, then the ring associated to $v$ under Theorem~\ref{par2} is $R\otimes\RR$.
 Explicitly, the multiplication tables of the algebras $R$ and $S$ with respect to the bases $\alpha$ and $\beta$, for a vector $v\in V$, are given by the same integer polynomial formulas in the coordinates of $v$ in the case of either Theorem~\ref{par} or \ref{par2}, namely, by \cite[\S 15 (1) and (2)]{DF} when $n=3$, by \cite[(14), (21), (22), and (23)]{III} when $n=4$, and by \cite[(16), (17), (21), and (22)]{IV} when $n=5$.

Now a nondegenerate ring $R$ of rank $n$ over $\RR$ must be a product of field extensions of~$\RR$; thus
$R\cong \RR^r\times\CC^s$ for some $r,s$ with $r+2s=n$.  Hence there are two nondegenerate orbits of $G_\RR$ on $V_\RR$ when $n=3$, three such orbits of $G_\RR$ on $V_\RR$ when $n=4$, and three such orbits of $G_\RR$ on $V_\RR$ when $n=5$, corresponding to the rings over $\RR$ given by
\begin{equation}\label{realringlist}
\RR^3, \; \RR\times\CC; \;\; \RR^4, \; \RR^2\times\CC, \; \CC^2; \;\; \RR^5, \; \RR^3\times\CC, \; \RR\times \CC^2,
\end{equation}
respectively. 
An explicit computation with these eight nondegenerate orbits arising in Theorem~\ref{par2}, or an elementary group theory argument, then shows that
 the corresponding quadratic, cubic, and sextic resolvent rings of the rings in (\ref{realringlist}) are given by
 \begin{equation}\label{resringlist}
 \RR^2, \; \CC; \;\; \RR^3, \; \RR\times\CC, \; \RR^3; \;\; \RR^6, \; \CC^3, \; \RR^2\times \CC^2,
 \end{equation}
 respectively.
  
% nondegenerate cubic rings over $\RR$ are $\RR^3$ and $\RR\times\CC$, which have quadratic resolvents $\RR^2$ and $\CC$, respectively; the nondegenerate quartic rings over $\RR$ are $\RR^4$, \,$\RR^2\times\CC$, and $\CC^2$, which have cubic resolvent rings given by $\RR^3$, $\RR\times\CC$, and $\RR^3$, respectively; and the nondegenerate quintic rings  over $\RR$ are $\RR^5$, \,$\RR^3\times\CC$, and $\RR\times \CC^2$, which have sextic resolvents given by $\RR^6$, $\CC^3$, and $\RR^2\times\CC^2$, respectively.

 Since a nondegenerate element of $v\in V_\RR$ determines a
 nondegenerate ring $R\cong \RR^r\times\CC^s$ for some $r,s$ with
 $r+2s=n$ (together with an $\RR$-basis $\alpha$ of $R/\RR$), we
 obtain a positive definite quadratic form $q_v$, defined by
 (\ref{qdef1}), on the trace zero $\RR$-subspace of $R$; here
 $\sigma_1,\ldots,\sigma_r$ and $\tau_1,\tau_1',\ldots,\tau_s,\tau_s'$
 denote as before the ring homomorphisms from $R$ to $\RR$ and the
 complex conjugate pairs of ring homomorphisms from $R$ to $\CC$,
 respectively.  The unique lift of the basis $\alpha$ to a basis of
 the trace zero subspace of $R$ makes $q_v$ an $(n-1)$-ary quadratic
 form.  By the last sentence of Theorem~\ref{par2}, this map $v\mapsto
 q_v$ from nondegenerate elements of $V_\RR$ to nondegenerate
 $(n-1)$-ary quadratic forms over $\RR$ is equivariant with respect to
 the action of $\gl_{n-1}(\RR)$ (and thus also equivariant with
 respect to the action of the subgroup $\gl_{n-1}(\ZZ)$).  \

 Finally, suppose $v\in V_\ZZ\subset V_\RR$ corresponds via
 Theorem~\ref{par} to the ring of integers $\OO_K$ of a number field
 $K$, together with a $\ZZ$-basis $\alpha$ of $\OO_K/\ZZ$.  Then by
 the compatibility of Theorems~\ref{par} and \ref{par2}, we see that
 the $(n-1)$-ary quadratic form $q_v$ defined in the previous
 paragraph is the same (up to a factor of $n$) as the quadratic form
 on the trace zero part of $\ZZ+n\OO_K$ defined in the introduction
 (with respect to the unique lift of the basis $n\alpha$ to a basis of
 the trace zero part of $\ZZ+n\OO_K$).  Therefore, if $v\in V_\ZZ$
 corresponds via Theorem~\ref{par} to the ring of integers~$\OO_K$ in
 a number field $K$, then the quadratic form $q_v$---as defined above
 for all vectors in $V_\RR$ using Theorem~\ref{par2}---gives the shape
 of $\OO_K$.

\section{Counting}\label{sec:3}

For $i\in\{ 0, 1, \ldots, \lfloor n/2 \rfloor \},$ let $V_\RR^{(i)}$
denote the subset of $V_\RR$ such that in the corresponding pair
$(R,S)$, the ring $R\otimes\RR$ of rank $n$ over $\RR$ is isomorphic
to $\RR^{n-2i} \times \mathbb{C}^i.$ Then $V_\RR^{(0)}, V_\RR^{(1)},
\ldots, V_\RR^{(\lfloor n/2 \rfloor)}$ are the nondegenerate orbits of
$G_\RR$ on $V_\RR$.  

The representation of $G_\RR=\gl_{n-1}(\RR)\times\gl_{r-1}(\RR)$ on $V_\RR$ is not faithful---indeed, the kernel is infinite.  The action of $G_\RR$ on $V_\RR$ 
factors through that of $G_\RR'=
\mathbb{G}_m(\RR)\times \gl^{\pm1}_{n-1}(\RR) \times \gl^{\pm1}_{r-1}(\RR)$ (where $\mathbb G_m$ acts by scalar multiplication), via $$(g_{n-1},~g_{r-1})\mapsto 
%\begin{cases} 
%\left(|\!\det g_{n-1}'|^{3} |\!\det g_{r-1}' |^{1},~g_{n-1}^\prime,~g_{r-1}^\prime\right) & \text{if $n=3$,}
%\\ 
%\left(|\!\det g_{n-1}'|^{2} |\!\det g_{r-1}' |^{1},~g_{n-1}^\prime,~g_{r-1}^\prime\right) & \text{if $n=4$,}
%\\ 
%\left(|\!\det g_{n-1}'|^{1} |\!\det g_{r-1}' |^{2},~g_{n-1}^\prime,~g_{r-1}^\prime\right) & \text{if }n=5, 
%\end{cases}
\begin{cases} 
\left(|\!\det g_{n-1}|^{3/(n-1)} |\!\det g_{r-1} |^{1/(r-1)},~g_{n-1}^\prime,~g_{r-1}^\prime\right) & \text{if $n=3$,}
\\ 
\left(|\!\det g_{n-1}|^{2/(n-1)} |\!\det g_{r-1} |^{1/(r-1)},~g_{n-1}^\prime,~g_{r-1}^\prime\right) & \text{if $n=4$,}
\\ 
\left(|\!\det g_{n-1}|^{1/(n-1)} |\!\det g_{r-1} |^{2/(r-1)},~g_{n-1}^\prime,~g_{r-1}^\prime\right) & \text{if }n=5, 
\end{cases}
$$
where $g_i^\prime$ is given by $g_i= |\!\det g_i|^{1/i} g_i^\prime;$ here for a matrix group $G$, we use $G^{\pm1}$ to denote the subgroup of $G$ consisting of those elements having determinant
$\pm1$.  
%The shape of $gv^{(i)}$ for $g \in G_\RR'$ is then simply the image of $g$ in $$\gl_{n-1}^{\pm1}(\mathbb{Z}) \backslash \gl^{\pm1}_{n-1}(\mathbb{R}) / \go^{\pm1}_{n-1}(\mathbb{R})\;\cong\;\gl_{n-1}(\mathbb{Z}) \backslash \gl_{n-1}(\mathbb{R}) / \go_{n-1}(\mathbb{R}).$$
The orbits of $G_\RR$ on $V_\RR$ are the same as the orbits of $G_\RR'$ on $V_\RR$, and the orbits of $G_\ZZ$ on $V_\ZZ$ are the same as the orbits of $G_\ZZ'$ on $V_\ZZ$,
where $G_\ZZ'=\mathbb{G}_m(\ZZ)\times \gl^{\pm1}_{n-1}(\ZZ) \times \gl^{\pm1}_{r-1}(\ZZ)$.  Furthermore, it is easy to see that the kernel of the representation of $G_\RR'$ (or $G_\ZZ'$) on $V_\RR$ is now finite and in fact of size 4.  In particular, the action of $G_\RR'$ (and thus $G_\ZZ'$) on $V_\RR$ has generically finite stabilizers.   Hence, in the sequel, it will often be convenient to refer to the actions of $G_\RR'$ and $G_\ZZ'$ on $V_\RR$ rather than those of $G_\RR$ and $G_\ZZ$, especially in situations where the finiteness of the stabilizer is important. 

For $i \in \{ 0, 1, \ldots, \lfloor n/2 \rfloor
\},$ let $v^{(i)} \in V_\RR^{(i)}$ be any fixed vector whose
associated shape $q(v^{(i)})$ in $\mc{S}_{n-1}$
%, defined to be the shape of the corresponding ring of integers, 
is $\left[\begin{smallmatrix} 1\! & & \\[-.07in] & \ddots & \\ & & \!1
    \\ \end{smallmatrix}\right]$.  To obtain such a $v^{(i)}$, choose
any $w^{(i)} \in V_\RR^{(i)}$; its shape is some positive definite
$(n-1)$-ary quadratic form $q(w^{(i)})$.  Since $\gl_{n-1}(\RR)$ acts
transitively on shapes in $\mc{S}_{n-1}$, there exists $\gamma \in
\gl_{n-1}(\RR)$ such that $\gamma \cdot q(w^{(i)}) =
\left[\begin{smallmatrix} 1\! & & \\[-.07in] & \ddots & \\ & & \!1
    \\ \end{smallmatrix}\right]$.  Let $v^{(i)} = \gamma \cdot
w^{(i)}$, where we view $\gl_{n-1}(\RR)
=\gl_{n-1}(\RR)\times\{1\}\subset G_\RR; $ then $q(v^{(i)}) = q(\gamma
\cdot w^{(i)}) = \gamma \cdot q(w^{(i)}) = \left[\begin{smallmatrix}
    1\! & & \\[-.07in] & \ddots & \\ & & \!1
    \\ \end{smallmatrix}\right]$.  By scaling $v^{(i)}$ by an element
of $\RR^\times$ if necessary, we may furthermore assume that
$|\disc(v^{(i)})|=1$.

Let $4n_i$ denote the cardinality of the stabilizer in $G_\RR'$ of
$v^{(i)} \in V_\RR^{(i)}.$ Let $\mc{F}\subset G_\RR'$ be a fundamental
domain for the left action of $G_\ZZ'$ on $G_\RR'$. (For instance, we
may take $\mc{F}$ to lie in a standard Siegel set; see
\cite{BH,DD4,DD5} for details.) Since the kernel of the representation of either $G_\RR'$ or $G_\ZZ'$ on $V_\RR$ has size 4, we see that $\mc{F}v^{(i)}$, viewed as a
multiset, is the union of $n_i$ fundamental domains for the action of
$G_\ZZ'$ on $V_\RR^{(i)}$.

Let $V_\ZZ^{(i)} := V_\RR^{(i)} \cap V_\ZZ$, and for any $G_\ZZ$-invariant subset $S
\subset V_\ZZ^{(i)}$, let $N(S;X)$ denote the number of 
$G_\ZZ$-orbits of 
irreducible elements 
%$x \in \mc{F}v^{(i)} \cap S$ 
$x\in V_\ZZ^{(i)}$ such that $|\disc(x)| < X.$ For
any nice subset $W \subset \mc{S}_{n-1}$, let $N(S; X, W)$ denote the
number of 
$G_\ZZ$-orbits of 
irreducible elements 
%$x \in \mc{F}v^{(i)} \cap S$ 
$x\in V_\ZZ$ such that
$|\disc(x)| < X$ and the shape $q(x)$ of $x$ is in $W.$

Let $\mc{R}_X := \{ x\in \mc{F}v^{(i)} : |\disc(x)| < X\}$, and
$\mc{R}_{X,W} := \{ x \in \mc{F}v^{(i)} : |\disc(x)| < X \mbox{ and }
q(x) \in W \}$, and let $\vol(\mc{R}_X)$ (respectively
$\vol(\mc{R}_{X,W})$) denote the Euclidean volume of $\mc{R}_X$
(respectively $\mc{R}_{X,W}$) as a multiset.

In \cite{DavI, DavII, DD4, DD5}, it is shown that (noting the slightly
different notions of irreducibility in these references that only
differ on negligible sets and thus do not matter here; see
(\cite[p.~183]{DavI}, \cite[p.~1037]{DD4}, \cite[p.~1583]{DD5}):
\begin{theorem}\label{thm:3}
$N(V_\ZZ^{(i)};X) = \displaystyle{\frac1{n_i}}\vol(\mc{R}_X) + o(X) = \displaystyle{\frac1{n_i}}\vol(\mc{R}_1)\cdot X + o(X).$
\end{theorem} 
In fact, the number of lattice points in $\mc{R}_X$ is $\gg
\vol(\mc{R}_X)$ as $X \rightarrow \infty$, but when only counting the
\emph{irreducible} lattice points, the number of such irreducible
points is $\sim \vol(\mc{R}_X)$ as $X \rightarrow \infty.$ This is an
important point and crucial to our proof of Theorem~\ref{thm:main}.

We use Theorem \ref{thm:3} in order to prove the following refinement:

\begin{theorem}\label{thm:4}
$N(V_\ZZ^{(i)};X,W) = \displaystyle{\frac1{n_i}}\vol(\mc{R}_{X,W}) + o(X) = \displaystyle{\frac1{n_i}}\vol(\mc{R}_{1,W})\cdot X + o(X).$
\end{theorem}

\begin{proof}  
	
	We will require the following lemma. 
		\begin{lemma}\label{lem:5}
                  If $H$ is any bounded measurable set in $V_\RR$,
                  then the number of irreducible lattice points in
                  $zH$ is $\vol(zH) + o(z^d)$ as $z \rightarrow
                  \infty$ $($i.e., reducible lattice points become
                  negligible as $z$ goes to infinity$)$.
		\end{lemma}

		\begin{proof}
                  Without the irreducibility conditions, this
                  statement is just the theory of Riemann integration,
                  though it can also be deduced via Davenport's Lemma
                  \cite{Dav}, which improves the $o(z^d)$ to
                  $O(z^{d-1}).$ Thus to obtain the lemma, it remains
                  to show that the reducible lattice points have
                  density 0 among all lattice points in $zH$ as
                  $z\to\infty$.  To show this, we may use the fact
                  that, if a lattice point in $V_\ZZ$ is reducible,
                  then it must satisfy various congruence conditions;
                  it then suffices to show that the density of points
                  satisfying all these congruence conditions is 0.
                  This follows in the case $n=5$ from \cite[\S
                  2.4]{DD5}; the cases $n=3$ and $n=4$ can be handled
                  in the identical manner (using the fact that $S_3$
                  is generated by any $2$-cycle and $3$-cycle, and
                  $S_4$ is generated by any $3$-cycle and $4$-cycle).
                  A more quantitative version of the density 0
                  statement for $n\in\{3,4\}$ can be deduced from
                  \cite[\S\S4,5]{DavI} and \cite[\S 2.4]{DD4}; this
                  allows us to replace the $o(z^d)$ term in the lemma
                  by $O(z^{d-1})$ in these cases.
%		Estimates on reducible points in $\mc{R}_X$ can be found in \cite[\S\S4,5]{Dav I} , \cite[\S 2.4]{DD4}, \cite[\S 2.4]{DD5}.  
%In each case, it is found that reducible points tend to be found on the ``boundaries'' (where certain coefficients are zero), and the total number of points with certain very small coefficients is always $O(X^\delta)$ for some $\delta < 1$ depending on which case is being looked at.  Since $\mc{R}_X=X^{1/d}\mc{R}_1$ and $H$ can be covered by finitely many $G_\ZZ$ translates of $R_1$, we find that the number of irreducible lattice points in $zH$ is $\vol(zH) + o(z^d) + O(z^{d\delta}) = \vol(zH) + o(z^d).$
		\end{proof}

                Let $\mc{R}'_{1,W}$ be a bounded, measurable subset of $\mc{R}_{1,W}$ such that 
                $\vol(\mc{R}'_{1,W}) \geq \vol(\mc{R}_{1,W}) -
                \epsilon,$ and let $\mc{R}'_{X,W} := X^{1/d} \cdot \mc{R}'_{1,W}$. 
                Lemma \ref{lem:5} then implies that
                the number of irreducible
                lattice points in $\mc{R}'_{X,W}$ is equal to
                $\displaystyle{\frac1{n_i}}\vol(\mc{R}'_{1,W})\cdot X
                + o(X)$. Therefore $N(V_\ZZ^{(i)}; X, W) \geq
                \#\{\mbox{irreducible lattice points in }\mc{R}'_{X,W}\} \geq
                \displaystyle{\frac1{n_i}}(\vol(\mc{R}_{1,W}) -
                \epsilon)\cdot X + o(X).$ Since this is true for any
                $\epsilon$, we conclude that
%                Let $W'$ be a bounded measurable subset of $W$ with
%                boundary of measure zero such that
%                $\vol(\mc{R}_{1,W'}) \geq \vol(\mc{R}_{1,W}) -
%                \epsilon.$ Since $W'$ is bounded, so is
%                $\mc{R}_{1,W'}$.  Lemma \ref{lem:5} then implies that
%                $N(V_\ZZ^{(i)}; X, W'),$ the number of irreducible
%                lattice points in $\mc{R}_{X,W'}$, is equal to
%                $\displaystyle{\frac1{n_i}}\vol(\mc{R}_{1,W'})\cdot X
%                + o(X)$.  Therefore $N(V_\ZZ^{(i)}; X, W) \geq
%                N(V_\ZZ^{(i)}; X, W') \geq
%                \displaystyle{\frac1{n_i}}(\vol(\mc{R}_{1,W}) -
%                \epsilon)\cdot X + o(X).$ Since this is true for any
%                $\epsilon$, we conclude that
\begin{equation} \label{eqn:1}
N(V_\ZZ^{(i)}; X,W) \geq \displaystyle{\frac1{n_i}}\vol(\mc{R}_{1,W}) \cdot X + o(X).
\end{equation}

Let $\overline{W} = \mc{S}_{n-1} - W.$  Running the same argument above with $\overline{W}$ in place of $W$, we have
\begin{equation}\label{eqn:2}
N(V_\ZZ^{(i)}; X, \overline{W}) \geq \displaystyle{\frac1{n_i}}\vol(\mc{R}_{1, \overline{W}}) \cdot X + o(X).
\end{equation}

Adding \eqref{eqn:1} and \eqref{eqn:2} , we obtain $$N(V_\ZZ^{(i)}; X, W) + N(V_\ZZ^{(i)}; X, \overline{W}) \geq \displaystyle{\frac1{n_i}}\vol(\mc{R}_{1, W})\cdot X + \displaystyle{\frac1{n_i}}\vol(\mc{R}_{1, \overline{W}})\cdot X + o(X),$$ 
i.e., 
\begin{equation}
\label{eqn:3}
N(V_\ZZ^{(i)}; X) \geq \displaystyle{\frac1{n_i}}\vol(\mc{R}_1) \cdot X + o(X).
\end{equation}

However, by Theorem \ref{thm:3}, we have equality in \eqref{eqn:3}, and therefore must also have equality in \eqref{eqn:1} and \eqref{eqn:2}. This concludes the proof of Theorem \ref{thm:4}.
\end{proof}

In order to prove equidistribution, we require the following result whose proof we defer to \textsection\ref{sec:Volumes}.

\begin{theorem}\label{thm:6}
For $n\in\{3,4,5\}$, we have
$$\frac{\vol(\mc{R}_{1, W})}{\vol(\mc{R}_1)} = \frac{\mu(W)}{\mu(\mc{S}_{n-1})}.$$
  \end{theorem}
 
 Granting this theorem for the moment, we then obtain 
 
\begin{corollary}
  When irreducible elements in $V_\ZZ^{(i)}$, up to
  $G_\ZZ$-equivalence, are ordered by discriminant, the shapes of
  these elements are equidistributed in $\mc{S}_{n-1} $ with respect
  to $\mu$.
\end{corollary}

\section{Congruence Conditions}
Let $S$ be any subset of $V_\ZZ$ defined by \emph{finitely many}
congruence conditions modulo prime powers.  Then in \cite{DH, DD4,
  DD5}, the following congruence version of Theorem \ref{thm:3} is
proven:

\begin{theorem}\label{thm:7}
For $n\in\{3,4,5\}$, we have
$$N(S;X) = \displaystyle{\frac1{n_i}}\prod_p \mu_p(S) \cdot \vol(\mc{R}_1) \cdot X + o(X),$$
where $\mu_p(S)$ denotes the $p$-adic density of $S$ in $V_\ZZ$.
\end{theorem}

We also have the following congruence version of Lemma \ref{lem:5}, whose proof is identical:

\begin{lemma}\label{lem:8}
If H is any bounded measurable subset of $V_\RR$, then the number of irreducible lattice points in $S\cap zH$ is $\displaystyle{\frac1{n_i}}\prod_p \mu_p(S) \cdot \vol(zH) + o(z^d)$ as $z \rightarrow \infty$.
\end{lemma}

Running the same argument in \S \ref{sec:3} with $S$ instead of $V_\ZZ$ (noticing that $S$ is the disjoint union of finitely many translates of lattices), we obtain

\begin{theorem}\label{thm:9}
For $n\in\{3,4,5\}$, we have
\begin{equation}
\label{eqn:4}
N(S; X, W) = \displaystyle{\frac1{n_i}}\prod_p \mu_p(S) \cdot \vol(\mc{R}_{1, W}) \cdot X + o(X).
\end{equation}
\end{theorem}

Again, using Theorem \ref{thm:6} we then obtain the following corollary about equidistribution.

\begin{corollary}
When irreducible elements of $S\cap V_\ZZ^{(i)}$, up to
  $G_\ZZ$-equivalence, are ordered by discriminant, the shapes of these elements are equidistributed in $\mc{S}_{n-1} $ with respect to $\mu$.
\end{corollary}

\section{Maximality}
Let $U$ denote the subset of elements of $V_\ZZ$ corresponding to $(R,S)$ where $R$ is a maximal ring of rank $n$, and $U_p$ the subset of  elements in $V_\ZZ$ where $R$ is maximal at $p.$  Then $U= \cap_p U_p$ and is defined by infinitely many congruences modulo prime powers (\cite[\S 2]{DH}, \cite[\S 4.10]{III}, \cite[\S 12]{IV}).  To show that \eqref{eqn:4} holds even for $S=U$, we require the following lemma, which is \cite[\S 4, Proposition 1]{DH}, \cite[Proposition 23]{DD4}, \cite[Proposition 19]{DD5} :

\begin{lemma} Let $W_p = V_\ZZ - U_p$.  Then $N(W_p; X) = O(X/p^2).$
\end{lemma}

Let $Y$ be any positive integer.  Then $\cap _{p<Y} U_p$ is defined by finitely many congruence conditions.  So, by Theorem \ref{thm:9}, we have
$$N(\bigcap_{p<Y} U_p; X, W) = \displaystyle{\frac1{n_i}}\prod_{p<Y} \mu_p(U_p) \cdot \vol(\mc{R}_{1,W}) \cdot X + o(X).$$
Then $$N(U; X, W) \leq N(\bigcap_{p<Y} U_p; X, W) = \displaystyle{\frac1{n_i}}\prod_{p<Y} \mu_p(U_p) \cdot \vol(\mc{R}_{1,W}) \cdot X + o(X).$$  Letting $Y$ tend to infinity, we obtain $N(U; X, W) \leq \displaystyle{\frac1{n_i}}\prod_p \mu_p(U_p) \cdot  \vol(\mc{R}_{1,W}) \cdot X + o(X).$  

To obtain the reverse inequality, we note $$ \bigcap_{p<Y} U_p \subset (U \cup \bigcup_{p \geq Y} W_p).$$
Therefore,
\begin{align*}
	N(U; X, W) &\geq N(\bigcap_{p<Y} U_p; X, W) - \displaystyle{\sum_{p\geq Y}} N(W_p; X, W)\\
			&= \displaystyle{\frac1{n_i}}\prod_{p<Y} \mu_p(U_p) \cdot \vol(\mc{R}_{1,W}) \cdot X + o(X) - \displaystyle{\sum_{p\geq Y}} O(X/p^2) \\
			&= \displaystyle{\frac1{n_i}}\prod_{p<Y} \mu_p(U_p) \cdot \vol(\mc{R}_{1,W}) \cdot X + o(X) + O(X)\cdot\displaystyle{\sum_{p \geq Y}} 1/p^2 \\
			&= \displaystyle{\frac1{n_i}}\prod_{p<Y} \mu_p(U_p) \cdot \vol(\mc{R}_{1,W}) \cdot X + o(X) + O(X/Y).
\end{align*}
Letting $Y$ tend to infinity, we obtain $$N(U; X, W) = \displaystyle{\frac1{n_i}}\prod_p \mu_p(U_p) \cdot \vol(\mc{R}_{1,W}) \cdot X + o(X).$$
Together with Theorem \ref{thm:6}, this implies Theorem \ref{thm:main}.

\section[Computation of volumes]{Computation of volumes: Proof of Theorem \ref{thm:6}}\label{sec:Volumes}

In this section, we prove Theorem \ref{thm:6}, namely, that
\begin{equation}\label{volstocompute}
 \frac{\vol(\mc{R}_{1,W})}{\vol(\mc{R}_1)} = \frac{\mu(W)}{\mu(\mc{S}_{n-1})}.
 \end{equation}

%Before we 
To 
calculate the volumes occurring in (\ref{volstocompute}), we note 
%two things.  First, 
that 
since $v^{(i)}$ has shape $\left[\begin{smallmatrix} 1\! & & \\[-.07in]  & \ddots & \\  & & \!1 \\ \end{smallmatrix}\right]$, the shape of $gv^{(i)}$ for $g \in G_\RR$ is simply the image of $g$ in $$\gl_{n-1}(\mathbb{Z}) \backslash \gl_{n-1}(\mathbb{R}) / \go_{n-1}(\mathbb{R}).$$  
%That is, if $q$ is a shape and $q = g \cdot \left[\begin{smallmatrix} 1\! & & \\[-.07in]  & \ddots & \\  & & \!1 \\ \end{smallmatrix}\right] \cdot g^t,$ for $g \in \gl_{n-1}(\RR)$, then we associate to $q$ the image of $g$ in $\gl_{n-1}(\mathbb{Z}) \backslash \gl_{n-1}(\mathbb{R}) / \go_{n-1}(\mathbb{R})$. 
Similarly,
%Second, we 
%note that the action of $G_\RR=\gl_{n-1}(\RR)\times\gl_{r-1}(\RR)$ on $V_\RR$ is not faithful, but factors through that of $G_\RR'=
%\mathbb{G}_m(\RR)\times \gl^{\pm1}_{n-1}(\RR) \times \gl^{\pm1}_{r-1}(\RR)$ (where $\mathbb G_m$ acts by scalar multiplication), via $$(g_{n-1},~g_{r-1})\mapsto 
%%\begin{cases} 
%%\left(|\!\det g_{n-1}'|^{3} |\!\det g_{r-1}' |^{1},~g_{n-1}^\prime,~g_{r-1}^\prime\right) & \text{if $n=3$,}
%%\\ 
%%\left(|\!\det g_{n-1}'|^{2} |\!\det g_{r-1}' |^{1},~g_{n-1}^\prime,~g_{r-1}^\prime\right) & \text{if $n=4$,}
%%\\ 
%%\left(|\!\det g_{n-1}'|^{1} |\!\det g_{r-1}' |^{2},~g_{n-1}^\prime,~g_{r-1}^\prime\right) & \text{if }n=5, 
%%\end{cases}
%\begin{cases} 
%\left(|\!\det g_{n-1}|^{3/(n-1)} |\!\det g_{r-1} |^{1/(r-1)},~g_{n-1}^\prime,~g_{r-1}^\prime\right) & \text{if $n=3$,}
%\\ 
%\left(|\!\det g_{n-1}|^{2/(n-1)} |\!\det g_{r-1} |^{1/(r-1)},~g_{n-1}^\prime,~g_{r-1}^\prime\right) & \text{if $n=4$,}
%\\ 
%\left(|\!\det g_{n-1}|^{1/(n-1)} |\!\det g_{r-1} |^{2/(r-1)},~g_{n-1}^\prime,~g_{r-1}^\prime\right) & \text{if }n=5, 
%\end{cases}
%$$
%where $g_i^\prime$ is given by $g_i= |\!\det g_i|^{1/i} g_i^\prime.$
%%$g_i'= |\!\det g_i|^{-1/i} g_i.$
%%The action of $G_\RR'$ on $V_\R$ generically has finite stabilizers.
%Again, 
the shape of $gv^{(i)}$ for $g \in G_\RR'$ is simply the image of $g$ in $$\gl_{n-1}^{\pm1}(\mathbb{Z}) \backslash \gl^{\pm1}_{n-1}(\mathbb{R}) / \go^{\pm1}_{n-1}(\mathbb{R})\;\cong\;\gl_{n-1}(\mathbb{Z}) \backslash \gl_{n-1}(\mathbb{R}) / \go_{n-1}(\mathbb{R}),$$
where again for any matrix group $G$ we use $G^{\pm1}$ to denote the subgroup of $G$ consisting of those elements having determinant
$\pm1$.  
%It is easy to see that the orbits of $G_\RR$ on $V_\RR$ are the same as the orbits of $G_\RR'$ on $V_\RR$, and the orbits of $G_\ZZ$ on $V_\ZZ$ are the same as the orbits of $G_\ZZ'$ on $V_\ZZ$,
%where $G_\ZZ'=\mathbb{G}_m(\ZZ)\times \gl^{\pm1}_{n-1}(\ZZ) \times \gl^{\pm1}_{r-1}(\ZZ)$.  Furthermore, the action of $G_\RR'$ (and thus $G_\ZZ'$) on $V_\RR$ has generically finite stabilizers.   
% rewrite $G_\RR$ as $\gl_1(\RR)\left(\SL_{r-1}(\RR) \times \SL_{n-1}(\RR)\right)$ and note that only $\SL_{n-1}(\RR)$ acts on the ring $R$, thus the shape of $gv^{(i)}$ depends only on the $\SL_{n-1}(\RR)$ component???? of $g.$

%Let $dg$ denote any fixed Haar measure on $G_\RR'$.  
We use the following proposition, which immediately follows from \cite[Proposition 2.4]{Sh}, \cite[Proposition~21]{DD4}, \cite[Proposition 16]{DD5} via an application of Lebesgue's dominated convergence theorem and the density of the bounded continuous functions in the integrable ones:

\begin{proposition}
For $i \in \{0, 1, \ldots, \lfloor n/2 \rfloor \},$ let $f \in L^1(V_\RR^{(i)}).$  Then there exists a nonzero constant $c_i$ such that $$\int_{g \in G_\RR'} f(g \cdot v^{(i)}) dg \,=\, c_i \cdot \int_{v \in V_\RR^{(i)}} |\disc(v)|^{-1} f(v) dv. $$
\end{proposition}

\noindent
(We always use $dg$ to denote a fixed Haar measure.)

In the above proposition, 
%take $f$ to be in turn the characteristic function $\chi_{\mc{R}_{1,W}}(g\cdot v^{(i)})$ of $\mc{R}_{1,W}$ and $\chi_{\mc{R}_1}(g\cdot v^{(i)})$ of $\mc{R}_1$. Then,
set $f(v)$ to be in turn $\chi_{\mc{R}_{1,W}}(v)|\disc(v)|$ and 
$\chi_{\mc{R}_{1}}(v)|\disc(v)|$, where $\chi_{\mc{R}_{1,W}}$ and $\chi_{\mc{R}_{1}}$ denote the characteristic functions of $\mc{R}_{1,W}$ and $\mc{R}_1$, respectively. Then,
\begin{equation}\label{volformula}
\frac{\vol(\mc{R}_{1,W})}{\vol(\mc{R}_1)} = \frac{\displaystyle\int_{g\in G_\RR'}\chi_{\mc{R}_{1,W}}(g\cdot v^{(i)})|\disc(g\cdot v^{(i)})|dg}{\displaystyle\int_{g\in G_\RR'}\chi_{\mc{R}_1}(g\cdot v^{(i)})|\disc(g\cdot v^{(i)})|dg}.
\end{equation}

Since $\mc{R}_1$ lies in a fundamental domain for the action of $G_\ZZ'$ on $V_\RR^{(i)}$,  the integrals appearing in 
(\ref{volformula}) may naturally be taken over 
%$g\in\mc{F}$. We may thus think of them as integrals over 
$G_\ZZ'\backslash G_\RR'$.
Now if $g=(\lambda,g_{n-1},g_{r-1})\in G_\RR'$, 
%can be written as $g=\lambda s$ where $s$ has ``determinant'' 1. Since 
then $\lambda$ does not affect the shape (i.e., $q((\lambda,g_{n-1},g_{r-1})\cdot v)=q((1,g_{n-1},g_{r-1})\cdot v)$).
Since $\disc(g\cdot v^{(i)})=\lambda^{d}$, the ratio (\ref{volformula}) becomes
$$\frac{\vol(\mc{R}_{1,W})}{\vol(\mc{R}_1)} = \frac{\displaystyle\int_0^1\lambda^{{d}}d^\times\lambda\int_{g\in G''_\ZZ\backslash G''_\RR}\chi_{\mc{R}_{1,W}}(g\cdot v^{(i)})dg}{\displaystyle\int_0^1\lambda^{{d}}d^\times\lambda\int_{g\in G''_\ZZ\backslash G''_\RR}\chi_{\mc{R}_1}(g\cdot v^{(i)})dg},$$
where we use $G_T''$ to denote simply $\gl^{\pm1}_{n-1}(T) \times \gl^{\pm1}_{r-1}(T)$.
% (and $dg$ is now a Haar measure on~$G_\RR''$). 
In fact, for any $g\in G_\RR'$ where $0<\lambda<1$, we have that $g\cdot v^{(i)}$ has absolute discriminant $<1$, and so the characteristic function $\chi_{\mc{R}_1}$ occurring in the integral in the denominator can be removed.
%For $n=4$ (resp.~$n=5$), the $\SL_2$ (resp.~$\SL_5$) 
Now the factor of $\gl^{\pm1}_{r-1}(T)$ also does not affect the shape, and thus
$$\frac{\vol(\mc{R}_{1,W})}{\vol(\mc{R}_1)} = \frac{\displaystyle\int_{g\in \gl^{\pm1}_{r-1}(\ZZ)\backslash \gl^{\pm1}_{r-1}(\RR)}dg\int_{g\in \gl^{\pm1}_{n-1}(\ZZ)\backslash \gl^{\pm1}_{n-1}(\RR)}\chi_{\mc{R}_{1,W}}(g\cdot v^{(i)})dg}{\displaystyle\int_{g\in \gl^{\pm1}_{r-1}(\ZZ)\backslash \gl^{\pm1}_{r-1}(\RR)}dg\int_{g\in \gl^{\pm1}_{n-1}(\ZZ)\backslash \gl^{\pm1}_{n-1}(\RR)}dg}.$$
Since $q(g\cdot v^{(i)})=
%gg^t$ for $g\, if $k\in SO_{n-1}(\RR)$, then $q(k\cdot v^{(i)})=
q(v^{(i)})$ for $g\in \go^{\pm1}(\RR)$,  we obtain
$$\frac{\vol(\mc{R}_{1,W})}{\vol(\mc{R}_1)} = \frac{\displaystyle
\int_{g\in \gl^{\pm1}_{n-1}(\ZZ)\backslash \gl^{\pm1}_{n-1}(\RR)/\go^{\pm1}_{n-1}(\RR)}\chi_{\mc{R}_{1,W}}(g\cdot v^{(i)})dg\int_{k\in\go^{\pm1}_{n-1}(\RR)}dk}
{\displaystyle\int_{g\in \gl^{\pm1}_{n-1}(\ZZ)\backslash \gl^{\pm1}_{n-1}(\RR)/\go^{\pm1}_{n-1}(\RR)}dg\int_{k\in\go^{\pm1}_{n-1}(\RR)}dk}.$$
Since the $g\in\gl^{\pm1}_{n-1}(\ZZ)\backslash \gl^{\pm1}_{n-1}(\RR)/\go^{\pm1}_{n-1}(\RR)$ such that $g\cdot v^{(i)}$ has shape in $W$ are exactly those that are in $W$, we obtain Theorem~\ref{thm:6}.

\subsection*{Acknowledgments}

We are extremely grateful to Rob Harron, Wei Ho, Hendrik Lenstra, Melanie Matchett Wood, Peter Sarnak, and Arul Shankar for many helpful conversations. 
It is also a pleasure to thank the Packard Foundation and the National Science Foundation (grant DMS-1001828) for their kind support.

\bibliographystyle{amsplain}
\bibliography{equiarxivfinal.bib}

%{array}{cccc}
%\begin{thebibliography}{aaaaa}
   %\bibitem{III} HCL III
   %\bibitem{IV} HCL IV
   %\bibitem{DD4} The Density of Discriminants of Quartic Rings and Fields
   %\bibitem{DD5} The Density of Discriminants of Quintic Rings and Fields
   %\bibitem{BhSh} Shnidman
   %\bibitem{BH} Borel--Harish-Chandra, {\it Annals of Mathematics}.
   %\bibitem{Dav} On a Principal of Lipschitz
   %\bibitem{DavI} On the Class Number of Binary Cubic Forms (I)
   %\bibitem{DF} Delone--Fadeev
   %\bibitem{DH} Davenport--Heilbronn
   %\bibitem{SK} Sato--Kimura
   %\bibitem{Sh} Shintani
   %\bibitem{Terr} Terr      
   %\bibitem{WY} Wright--Yukie
   %\end{thebibliography}

\end{document}